\documentclass[a4paper]{article}

\usepackage{amssymb}
\usepackage{graphicx,color}
\usepackage{amsthm}
\usepackage{amsmath}

\usepackage{float}
\usepackage{subfig}
\usepackage{hyperref}
\usepackage{multirow}
\usepackage{rotating}

\theoremstyle{plain}
\newtheorem{theorem}{Theorem}
\newtheorem{proposition}[theorem]{Proposition}

\newtheorem{lemma}[theorem]{Lemma}

\theoremstyle{definition}

\newtheorem{definition}[theorem]{Definition}

\def\ho{\H o}

\def\ra{\rightarrow}
\def\ds{\displaystyle}
\def\ba{\begin{array}}
\def\ea{\end{array}}
\def\bi{\begin{itemize}}
\def\ei{\end{itemize}}

\def\mE{\mathbb{E}}

\def\cX{\mathcal{X}}
\def\ds{\displaystyle}
\def\fixcu{\underline{c}}
\def\fixco{\overline{c}}

\begin{document}

\title{Mixing times of Markov chains on a cycle\\ with additional long range connections}

\author{Bal\'azs Gerencs\'er
\thanks{Universit\'e catholique de Louvain, \texttt{balazs.gerencser@uclouvain.be}. This paper presents research results partially supported by the Concerted
  Research Action (ARC) Large Graphs and Networks of the French
  Community of Belgium and the Belgian Network DYSCO (Dynamical
  Systems, Control, and Optimization), funded by the Interuniversity
  Attraction Poles Programme, initiated by the Belgian State, Science
  Policy Office.}
\thanks{On leave from the MTA R\'enyi Institute.}
}

\date{}

\maketitle

\begin{abstract}
We develop Markov chain mixing time estimates for a class of Markov
chains with restricted transitions. We assume transitions may occur
along a cycle of $n$ nodes and on $n^\gamma$ additional edges, where $\gamma
< 1$. We find that the mixing times of reversible Markov chains
properly interpolate between the mixing times of the cycle with no added
edges and
of
the cycle with $cn$ added edges (which is in turn a Small World
Network model). In the case of non-reversible Markov-chains, a
considerable gap remains between lower and upper bounds, but
simulations give hope to experience a significant speedup compared to the reversible case. 



\end{abstract}


\section{Introduction}

Mixing time is an important quantity arising in numerous
applications. In Markov Chain Monte Carlo (MCMC) simulations, as
described in the now classic papers Metropolis et
al. \cite{metropolis:alg} and Hastings \cite{hastings:mmixing}, 
mixing time can be interpreted as the time needed to generate a single
sample of a given distribution with prescribed accuracy. See Jerrum
\cite{jerrum:mcmc1998} for a modern exposition on the subject.

Another current, hot area of application is the the theory of
distributed average consensus algorithms. The flow of such an
algorithm can be viewed as the evolution of a distributions on some
states according to a Markov chain. For details see Olshevsky,
Tsitsiklis \cite{olsh_tsits:consensus_speed2009} or Boyd et
al. \cite{boyd_and_al:gossip2006}. Here the time needed to get within
a certain neighborhood of the average value can be quantified by the
mixing time. Motivated by these and other applications, the estimation
of mixing time is in the center of interest.

The present paper is linked to two previous results. First, a random
walk on a cycle of $n$ nodes has a mixing time of $cn^2$. This is the
result of the central limit theorem which tells us that we need $ck^2$
steps to move away to a distance of $k$ from the starting point.
Second, when random edges are added to the cycle with a density of $2n^{-1}$
resulting in roughly $n$ random edges, the mixing time drops to
$c\log^2 n$, see Durrett \cite{durrett:rgd}, Addario-Berry 
and Lei \cite{addarioberry:swnmixing2012}.

Our goal is to investigate the case in between, namely when the number
of added random edges is sublinear, with a density $2n^{-\alpha}$ for $\alpha \in
(1,2)$. This gives roughly $n^{2-\alpha}$ new edges. We find that the mixing time in the reversible case is
$n^{2-2\alpha}$ (up to logarithmic factors). For a non-reversible
random walk however, when symmetry is broken, the lower bound on the
mixing time drops to $n^{1-\alpha}$ (up to logarithmic factors again). Determining the exact value is
still open, but simulations indicate that a strong speedup is present
for this case. Precise description of the random graph models are
given in Definition \ref{DEF:MODELS}, the exact mixing time bounds are
presented in Theorem \ref{THM:M1HOMO} through \ref{THM:M3HOMO}.

The choice of the cycle as a base graph is justified by the following
two reasons. First, in the case of $n$ added random edges, this is
exactly the already well-known model of Newman et
al. \cite{newman:swn2000} for modeling Small World Networks. Second,
there is a very natural way of defining non-reversible random walks on
the graph models considered simply by introducing a drift along the
cycle, increasing transition probabilities in one direction and
decreasing them in the other. When the base graph is chosen to be
another connected graph, the case of $n$ random edges and reversible
random walks is treated in depth by Krivelevich, Reichman and
Samotij \cite{krivelevich2013smoothed}.

The rest of the paper is structured as follows. 
Section \ref{SEC:PRELIMINARIES} presents general definitions and tools
used in the paper.
In Section \ref{SEC:RANDEDGEMODELS} we describe the connectivity
graphs we work with. In
Section \ref{SEC:RANDCOND} we work out the intermediate estimates
required to complete our work.
The main results are deduced in Section \ref{SEC:RANDMIXTIME}. We
close with conclusions in Section \ref{SEC:CONC}.

\section{Preliminaries}
\label{SEC:PRELIMINARIES}

In order to be self-contained we present the definition of mixing time together with some related concepts.
We will work with aperiodic, irreducible Markov chains on a finite state space
$\cX$ which has size $n = |\cX|$. The set of probability distributions on the state space $\cX$
will be denoted by $\mathcal{P}(\cX)$.
In order to define mixing time we need a metric to measure the distance of probability
distributions.

One of the widely used options is the \emph{total
variation distance} defined as follows:
\begin{definition}
Given two probability measures $\mu, \nu$ on $\cX$, the \emph{total variation distance} is
defined as
$$\|\mu - \nu \|_{\rm TV}=\max_{A\subseteq\cX}|\mu(A)-\nu(A)|.$$
\end{definition}
When considering Markov chain the transition matrix is denoted by $P=(p_{ij})$, with $p_{ij}$
referring to
the probability of moving from state $i$ to state $j$, while
$\pi$ stands for the unique stationary distribution (if it exists).
We can now define the central notion of this paper.
\begin{definition}
For a Markov chain having a unique stationary
distribution we define the mixing time of the chain for any $\varepsilon > 0$ as
$$t_{\rm mix}(P,\varepsilon)=\max_{\sigma\in
  \mathcal{P}(\cX)}\min\left\{k:\|\sigma P^{k}-\pi\|_{{\rm
      TV}}\le\varepsilon\right\}.$$
\end{definition}
For the rest of the paper $\varepsilon$ is considered to be fixed
  thus we use the simplified notation $t_{\rm mix}(P)$ or even
  $t_{\rm mix}$ when the argument $P$ is obvious from the context.

Often we are not only interested in the behavior of a single Markov
chain, but also in the achievable performance by modifying the Markov chain while
keeping the structure. For this we use the following definition.
\begin{definition}
The \emph{connectivity graph} of a Markov chain is a graph on the states of
the Markov chain. We connect nodes $i\neq j$ if either $p_{ij}> 0$ or
$p_{ji} > 0$.
\end{definition}
We restrict ourselves to those cases where the unique stationary distribution is
uniform. For the transition matrix this translates to the condition of
being doubly stochastic.

We are also interested in the role of the symmetry property, \emph{reversibility}.
\begin{definition}
A Markov chain is \emph{reversible} if starting from the stationary
distribution $\pi$, the probability of the consecutive pair $(i,j)$ is
the same as the probability of the consecutive pair $(j,i)$. Formally:
$$\pi_ip_{ij}=\pi_jp_{ji}\quad\forall i,j.$$
\end{definition}
The usefulness of the separation of reversible and non-reversible
Markov chains is widely recognized in the literature, see e.g.,
Montenegro and Tetali \cite{montenegro_tetali:slowmixing2006}. Often
it is more convenient
to prove certain properties for reversible chains, and there are
tighter general bounds on the mixing time for them. The reason to consider also
non-reversible chains is the fact that they may deliver much faster
mixing than similar reversible chains.

A tool we heavily rely on as a proxy to the mixing time is the \emph{conductance} of a Markov chain, introduced by
Jerrum and Sinclair \cite{jerrum_sinclair:1988conductance}. This
is a quantity indicating the capacity of
the worst bottleneck of the chain when the state space is
split into two parts.
\begin{definition}
For any $A,B\subset \cX,~A\cap B=\emptyset$ we set
$$Q(A,B) = \sum_{i\in A,j\in B}\pi_i p_{ij},$$
representing the flow from $A$ to $B$ for the stationary distribution.
The \emph{conductance} of a Markov chain is defined as
$$\Phi=\min_{\emptyset\ne S\subsetneq
  \cX}\Phi(S) = \min_{\emptyset\ne S\subsetneq
  \cX}\frac{Q(S,S^C)}{\pi(S)\pi(S^C)}=\min_{\emptyset\ne S\subsetneq
  \cX}\frac{\sum_{i\in S,j\in S^c}\pi_ip_{ij}}{\pi(S)\pi(S^C)},$$
where $S^C=\cX\setminus S$, the complement of the set $S$.
\end{definition}
This neat concept has evolved since
its birth into different refined variants such as
average conductance (see Lov\'asz, Kannan
\cite{lovasz:1999faster}), and blocking conductance (see Kannan,
Lov\'asz, Montenegro \cite{kannanetal:conductance2006}).

The importance is in the fact that this geometric
quantity can be directly related to mixing times. The
lower bound is easy to verify:
\begin{proposition}
  \label{PRP:CONDMIXTIME1}
  There is a constant $c>0$ such that for any Markov chain we have
  $$c\frac{1}{\Phi} \le t_{\rm mix}.$$
  The constant $c$ depends only on $\varepsilon$, showing up in
  the definition of the mixing time.
\end{proposition}
It is also possible to deduce an upper bound, as seen by
Jerrum and Sinclair
\cite{sinclair_jerrum:rapidmixing1989}:
\begin{theorem}
  \label{THM:JSMIXTIME}
  There is a constant $c>0$ 
  such that for any aperiodic, irreducible,
  reversible Markov chain the following bound for the mixing time
  holds:
  $$t_{\rm mix} \le c\frac{1}{\Phi^2}\log\left(\frac{1}{\pi_*}\right),$$
  where $\pi_*$ refers to the lowest value of the stationary
  distribution, $\pi_* = \min_i \pi_i$. Currently the stationary
  distribution is uniform, thus $\pi_* = 1/n.$
  The constant $c$ depends only on $\varepsilon$, showing up in
  the definition of the mixing time.
\end{theorem}
Let us also cite the following version of the above
theorem due to Lov\'asz and Simonovits
\cite{lovasz_simonovits:conductance1990}. This theorem does not
require reversibility, but it
assumes that the Markov chain is \emph{lazy}, i.e., $p_{ii}\ge 1/2$ for all
$i$:
\begin{theorem}
  \label{THM:LOVSIMON}
  There is a constant $c>0$ 
  such that for any aperiodic, irreducible, lazy
  Markov chain the following bound for the mixing
  time holds:
  $$t_{\rm mix} \le
  c\frac{1}{\Phi^2}\log\left(\frac{1}{\pi_*}\right).$$
  The constant $c$ depends only on $\varepsilon$, showing up in the
  definition of the mixing time.
\end{theorem}

Based on these theorems we get an insight on the mixing properties of
the best Markov chains. Given a connectivity graph we can look for the
best reversible Markov chain having the lowest mixing time or even
relax the reversibility condition to get the fastest non-reversible one.
The possible gap between the mixing times of reversible and
non-reversible Markov chains is limited according to the following
proposition.
\begin{proposition}
\label{PRP:BESTREVNONREV}
For some fixed connectivity graph let $P$ and $\tilde{P}$ be the
doubly stochastic
transition matrices of the best reversible and non-reversible chains,
respectively, yielding the lowest mixing times. Then for the respective mixing times we have
$$t_{\rm mix}(P)\le  c t_{\rm mix}^2(\tilde{P})\log n.$$
\end{proposition}
\begin{proof}
Let us define $P'=(\tilde{P}+\tilde{P}^T)/2$. Knowing that the stationary distribution
is uniform it is easy to see that
$P'$ is the transition matrix of a reversible Markov chain with the
same connectivity graph, and the stationary distribution corresponding
to $P'$ is once again uniform.
Moreover, observe that $\Phi_{P'}(S)=\Phi_{\tilde{P}}(S)$ for any
$S\subset\cX$ thus $\Phi_{P'}=\Phi_{\tilde{P}}$.
Using Theorem \ref{THM:JSMIXTIME} and Proposition
\ref{PRP:CONDMIXTIME1} this implies
$$t_{\rm mix}(P')\le  c_1\frac{1}{\Phi_{P'}^2}\log n=
c_1\frac{1}{\Phi_{\tilde{P}}^2}\log n \le c_2 t_{\rm
  mix}^2(\tilde{P})\log n.$$
The matrix $P'$ might not be the best choice for a reversible
transition matrix, but substituting it with a better $P$ just further
decreases the left hand side.
\end{proof}

Clearly for a random graph we cannot completely exclude some pathological cases.
Therefore, we are interested in the typical behavior, and we look for properties that are true
\emph{asymptotically almost surely (a.a.s.)} as the size of the graph
goes to infinity.
We are interested in the order of the
mixing time as $n$ increases but we do not care about constant
factors. For that reason, we use $c$ or $c_i$ for constants whose
value is unimportant. They might represent different values in each expression.

\section{Graph models}
\label{SEC:RANDEDGEMODELS}

First we recall the graph models of previous works for reference. Then
we give the detailed definition of the graph models currently
investigated. As noted in the introduction, one of the starting points is the case
when the connectivity graph is a cycle with $n$ nodes. The mixing time of the symmetric random walk is of the order
of $n^2$. It is far more complicated to deal with the case when we
consider any Markov chain, including non-reversible ones. Still, the order of magnitude of the mixing
time does not decrease, as shown by the author
\cite{gb:ringmixing2011}:

\begin{theorem}
\label{THM:SLOWMIX}
Consider a Markov chain on a cycle with
$n$ nodes having a doubly stochastic transition matrix $P$.
Then, with some global constant $C>0$ we have
$$t_{{\rm mix}}(P,1/8)\ge Cn^2.$$
\end{theorem}

The other end of the spectrum is the case where approximately $cn$
random edges are added to the cycle (for some constant $c>0$). This
way we get a model of Small World Networks (SWN). Namely if we add an
Erd\ho s-R\'enyi random graph with edge density $c/n$
to the cycle we get the model of Newman et al. \cite{newman:swn2000}. This and other similar models
were built to model large real networks, see Watts, Strogatz
\cite{watts_strogatz:swn}, Bollob\'as, Chung \cite{bollobas1988diameter}.
There is an intensive research activity on SWNs, in particular
the mixing time of random walks on them has been widely
investigated, see Tahbaz-Salehi and Jadbabaie \cite{tsj:swn_mixing} or
Hovareshti, Baras and Gupta \cite{hovareshti2008average}.
The following result is due to Durrett \cite{durrett:rgd}, Addario-Berry
and Lei \cite{addarioberry:swnmixing2012}:

\begin{theorem}
Consider an $n$ node graph from the model of Newman et
al. \cite{newman:swn2000}. Then for the
symmetric random walk on this graph we have
$$c_1\log^2 n<t_{\rm mix}<c_2\log^2 n$$
asymptotically almost surely (a.a.s.) with some global constants $c_1,c_2>0$.
\end{theorem}
This is a huge gain in speed compared to the mixing time of $n^2$ for the
cycle alone. Similar results have been recently shown when the random
edges are added to other base graphs by Krivelevich, Reichman and
Samotij \cite{krivelevich2013smoothed} based on the work of
Fountoulakis and Reed \cite{fountoulakis2007faster}.

Currently we investigate graph models where a sublinear number of
extra edges are added to the cycle. Let us add an interesting note for context.
Our initial goal was to decrease the mixing time by adding a few more
edges to the connectivity graph, starting from a cycle. It is far from trivial to choose the
edges that help the most. Therefore we performed numerical optimization to get
the best setting of the new edges, but none of the resulting graphs did show
any symmetry or structure, but looked random to the human eye.
This drove us to choose the edges randomly. This choice turned out
to be fruitful as we get consistently low mixing times with high probability.

We call the newly added edges \emph{long range edges} to distinguish
them from the original ones (which connect nodes that are ``close'').
Let the target edge density of the added long range edges be $2n^{-\alpha}$
for some parameter $\alpha\in(1,2)$. We therefore expect
$n^{2-\alpha}$ extra edges. We introduce three models to realize this concept:

\begin{definition}
\label{DEF:MODELS}
We use the following three random graph models with long range edge
density $2n^{-\alpha}$ for $\alpha\in(1,2)$.
\begin{itemize}
\item[{\it M1}:] We take the $2\left\lceil n^{2-\alpha}\right\rceil$ almost
equidistant nodes $\left\{[in^{\alpha-1}/2],~0\le i <
2\lceil n^{2-\alpha}\rceil\right\},$ and add edges corresponding to a random
matching on them.
\item[{\it M2}:] From all possible long range edges we draw a
subset of size $\left\lceil n^{2-\alpha}\right\rceil$ randomly, uniformly.
\item[{\it M3}:] For all possible long range edge we randomly
decide to include it or not. Each edge is included independently
with probability $2n^{-\alpha}$.
\end{itemize}
\end{definition}

In the models M2 and M3 we allow original edges of the cycle
to be chosen as long range edges to simplify our discussion.

The coming results depend on the asymptotic growth rate of the number
of long range
edges, but not on whether we have exactly $\left\lceil
  n^{2-\alpha}\right\rceil$ or $\left\lfloor
  n^{2-\alpha}\right\rfloor$ of them. In this spirit we omit integer
rounding operations from now on, this only introduces an
asymptotically vanishing multiplicative error, but relieves
unnecessary complexity from our formulas.

We mainly consider the case of \emph{homogeneous} chains which
have a simple transition probability structure.
\begin{definition}
  \label{DEF:HOMOCHAINS}
  Fix some $q_c,q_l,r,d$.
  The \emph{homogeneous} Markov chain is
  defined by setting all clockwise transition probabilities on the
  cycle to $q_c+r$, and the counterclockwise ones to $q_c-r$. Long
  range edges are used with probability $q_l/d$. Otherwise the Markov
  chain stays put.

  We get the \emph{reversible homogeneous} Markov chain by setting
  $r=0$.

  These constructions do not necessarily give a proper Markov
  chain, feasibility depending on the parameters is discussed below.
\end{definition}

A random graph may occasionally have
nodes with high degree. The role of $d$ is to prevent such
nodes from having extremely high outgoing transition probabilities. The
following theorem ensures that this way we define meaningful Markov chains.

\begin{theorem}
\label{THM:MCFEASIBLE}
For every $\alpha \in (1,2)$ there is a $d(\alpha)$ such that for all
3 classes of random graphs M1, M2, M3 there is no node
with more than $d(\alpha)$ long range edges a.a.s.\\
Consequently, assuming $2q_c+q_l\le 1$ and setting $d=d(\alpha)$, homogeneous chains will be
feasible Markov chains
a.a.s.
\end{theorem}
\begin{proof}
For graphs from model M1 the statement is straightforward: every node
has 0 or 1 long range edge.

Let us now consider
a graph from model M3. Take a single node and denote
the number of its long range edges by $X$. Clearly it
follows a binomial distribution $Binom(n-1,2n^{-\alpha})$. To get
an upper bound on $P(X>d(\alpha))$ we use a Chernoff-type estimate
$$P(X>d(\alpha))=P\left(e^{tX}>e^{td(\alpha)}\right)\le
\frac{\mE(e^{tX})}{e^{td(\alpha)}},$$
with arbitrary $t>0$. The moment generating function of $X$ is
$$\mE(e^{tX})=\left(1+2n^{-\alpha}(e^t-1)\right)^{n-1}.$$
Let us choose $t=(\alpha-1)\log n$ to get the following:
$$\mE(e^{tX})=\left(1+2n^{-\alpha}(n^{\alpha-1}-1)\right)^{n-1}\ra e^2$$
as $n\ra\infty$.
Now let us fix $d(\alpha) = 2/(\alpha-1)$ and any $c>e^2$.
For large enough $n$ we get
$$P(X>d(\alpha)) \le \frac{c}{e^{(\alpha-1)\log n\cdot d(\alpha)}} = \frac{c}{n^2}.$$
The probability of the event that some node has more than $d(\alpha)$
long range edges can be bounded from above by a simple union bound
resulting in $c/n$ which tends to $0$. With this the claim is proven.

For graphs from the model M2 the
number of long range edges of a single node follows a
hypergeometric distribution: out of $n \choose 2$ possible long range edges
$n^{2-\alpha}$ are marked, we count the number of those within the
$n-1$ possible edges of the current node.
This is less convenient to estimate than the binomial distribution
before. We will use our previously obtained bounds for M3 graphs by
showing a special way of generating an M2 graph.

We start with a modified ``heavy'' M3 graph where the edge
probability is increased to $4n^{-\alpha}$. Let the total number of
long range edges obtained be $m$. Depending
on whether $m$ exceeds $n^{2-\alpha}$ or not, we either discard some edges
chosen uniformly from the selected ones, or add some edges
chosen uniformly from the unselected ones. This way we get the
prescribed number of edges and by symmetry arguments it follows that the final subset is chosen uniformly
from all subsets of size $n^{2-\alpha}$.

We know that there is a $\tilde{d}(\alpha)$ such that the initial ``heavy'' M3
graph has at most $\tilde{d}(\alpha)$ long range edges at every node
a.a.s. If we have to discard edges from this graph then this
property remains true. We might increase the degree of a node only in
the case when $m$ is small and we have to add
edges. The probability of this to happen is:
$$P\left(m<n^{2-\alpha}\right) < P\left(\left|m-4n^{-\alpha}\frac{n(n-1)}{2}\right|>\frac{1}{2}n^{2-\alpha}\right)
< c\frac{n^24n^{-\alpha}(1-4n^{-\alpha})}{n^{4-2\alpha}} < cn^{\alpha-2}.$$
The first inequality is based on the inclusion of the events. To
control the deviation of $m$ from its expected value we use Chebyshev's inequality.

In the end, the probability on the left hand side also
vanishes as $n\ra\infty$, consequently using the value
of $\tilde{d}(\alpha)$ we got for ``heavy'' M3 graphs the statement of the theorem holds true for M2 graphs.
\end{proof}

The same way as we could ensure the feasibility of the random homogeneous
chains we can provide laziness.
If we choose $q_c,q_l$ such that
$2q_c+q_l\le\frac 1 2$, Theorem \ref{THM:MCFEASIBLE} shows that
the remaining probability to stay put is at least $\frac 1 2$ at
every node a.a.s.

\section{Conductance estimates}
\label{SEC:RANDCOND}

The next step is bounding the conductances of the Markov chains.
First we
present a technical tool to simplify the
minimization occurring at the calculation of the conductance.

\begin{lemma}
Suppose that $\emptyset \neq S_1, S_2 \subset \cX$, $S_1\cap S_2=\emptyset$ and
there is no edge between them. Then we have
$$\Phi(S_1\cup S_2) > \min(\Phi(S_1),\Phi(S_2)).$$
\end{lemma}
\begin{proof}
$$\Phi(S_1\cup S_2) = \frac{Q(S_1\cup S_2,(S_1\cup
  S_2)^C)}{\pi(S_1\cup S_2)\pi((S_1\cup S_2)^C)} =
\frac{Q(S_1,S_1^C)+Q(S_2,S_2^C)}{\pi(S_1)+\pi(S_2)} \cdot
\frac{1}{\pi((S_1\cup S_2)^C)}.$$
The first term is between $Q(S_1,S_1^C)/\pi(S_1)$ and
$Q(S_2,S_2^C)/\pi(S_2)$. The second term is strictly
greater than both $1/\pi(S_1^C)$ and $1/\pi(S_2^C)$, thus the
lemma follows.
\end{proof}
We immediately get the following property for the minimizing set.
\begin{proposition}
\label{PRP:CONNCOND}
The set $S$ minimizing $\Phi(S)$
must be connected.
\end{proposition}
\begin{proof}
Let $S\subset \cX$ be a disconnected set, $S_1$ one of it's connected
components. We may use the previous lemma with $S_1$ and $S_2 = S\setminus
S_1$ to obtain that $S$ is not minimizing $\Phi(S)$.
\end{proof}

Let us now present three theorems to
determine the exact order of magnitude of the conductance for all three models.

\begin{theorem}
\label{THM:M1COND}
Consider a graph from model M1. The conductance of the homogeneous
chain on this graph satisfies the following inequality a.a.s.:
$$c_1d(\alpha)^{-1}n^{1-\alpha}<\Phi<c_2n^{1-\alpha}.$$
\end{theorem}
\begin{proof}
The upper bound is simple:
Let $A$ be one of the $n^{\alpha-1}/2$ long arcs without a
long range edge. We can use $\Phi(A)$ to bound the conductance:
$$\Phi=\min_{\emptyset\ne S\subsetneq V}\Phi(S) \le \Phi(A)
=
\frac{Q(A,A^C)}{\pi(A)\pi(A^C)} \le
\frac{2n^{-1}}{n^{\alpha-2}/2\cdot 1/2}=cn^{1-\alpha}.$$

The lower bound is a bit more intricate. Using Proposition
\ref{PRP:CONNCOND} we have to minimize over connected subsets to find
$\Phi$. Connected subgraphs are composed of
a collection of arcs which are connected by long range edges. Let us
define a new chain with nodes $\tilde{\cX}$ as shown in
Figure \ref{FIG:GRAPHRED}. For every node of $\cX$ with a long range
edge there is one node in $\tilde{\cX}$. Two nodes of $\tilde{\cX}$ are connected if they are connected
in $\cX$ or if they follow each other on the cycle. In other words, we
reduce all long empty arcs to single edges. Clearly the new chain has
$2n^{2-\alpha}$ nodes. We use the same homogeneous transition
probabilities as before. 

\begin{figure}[ht]
\centering
\includegraphics[width=0.7\columnwidth]{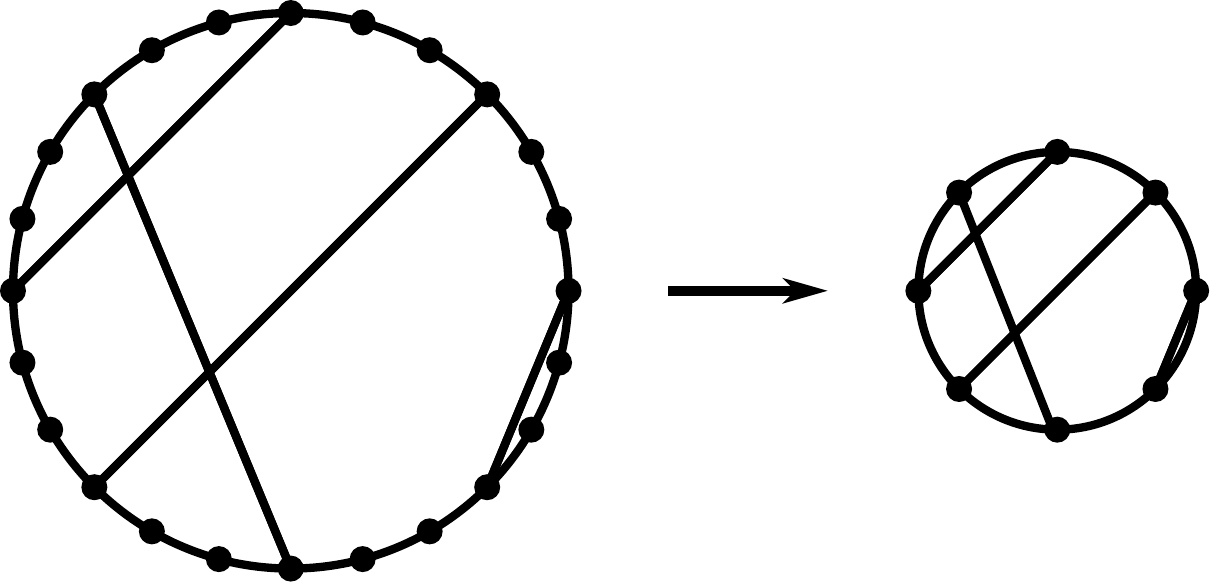}
\caption{Reducing M1 graphs}\label{FIG:GRAPHRED}
\end{figure}

We want to compare the conductance $\Phi$ of the original chain with the
conductance $\tilde{\Phi}$ of the new one. For any connected $S\subset
\cX$ we may naturally define
$\tilde{S}\subset\tilde{\cX}$ by keeping only the nodes in
$\tilde{\cX}$. When we want to bound $\Phi(S)$, we may freely swap $S$ with
$S^C$ as $\Phi(S)=\Phi(S^C)$. If $|\tilde{S}|>|\tilde{\cX}|/2$, let us
swap $S$ for $S^C$ (and pick one of its connected components if
needed). This way we can ensure $|\tilde{S}|\le|\tilde{\cX}|/2$. We
need to estimate the expressions appearing in $\Phi(S)$. The
transition probabilities are still $q_c\pm r$ along the cycle and
$q_l/d(\alpha)$ along long range edges. The stationary distribution is
uniform in both cases, but the number of nodes changes, so the
probability of individual points is scaled up by
$2n^{2-\alpha}/n$. Observe also that a boundary edge of $\tilde{S}$
corresponds to a boundary edge of $S$ with the same transition probability.
Therefore we get
$$Q(S,S^C) \ge \frac{2n^{2-\alpha}}{n}Q(\tilde{S},\tilde{S}^C).$$
For any node in $\tilde{S}$ there are at most the two adjacent empty
arcs present in $S$, consequently
$$\pi(S) < 2\tilde{\pi}(\tilde{S}).$$
For the complement set, we
made sure $\tilde{S}$ is
``small'' before so we have
$$\pi(S^C) < 1 \le 2\tilde{\pi}(\tilde{S}^C).$$
Combining these inequalities we arrive at
\begin{align}
\Phi(S)&\ge c n^{1-\alpha} \tilde{\Phi}(\tilde{S}),\nonumber \\
\Phi &\ge c n^{1-\alpha} \tilde{\Phi}.\label{EQ:M1PHIRED}
\end{align}

The reduced graph is a cycle with $2n^{2-\alpha}$ nodes with a random
matching  added, which is exactly the Bollob\'as-Chung small world
model \cite{bollobas1988diameter}. The conductance of the symmetric
random walk on the 
Bollob\'as-Chung model is already known, see e.g.\ Durrett \cite{durrett:rgd} p.\ 163-164., where it is shown that
it is bounded below by a positive constant. Our reduced chain is slightly
different as the long range edges have transition probabilities
$q_l/d(\alpha)$ instead of a global constant. The conductance scales
with the transition probabilities, hence for our reduced chain we
have
\begin{equation}
\label{EQ:TILDEPHIBOUND}
\tilde{\Phi}\ge c d(\alpha)^{-1}.
\end{equation}
Using this bound together with Equation \ref{EQ:M1PHIRED} completes the proof.
\end{proof}

\begin{theorem}
\label{THM:M2COND}
Consider a graph from model M2. The conductance of the homogeneous
chain on this graph satisfies the following inequality a.a.s.:
$$c_1d(\alpha)^{-1}\frac{n^{1-\alpha}}{\log n}<\Phi<c_2\frac{n^{1-\alpha}}{\log n}.$$
\end{theorem}
\begin{proof}
To establish an upper bound, we search again for a long arc $A$ without a long
range edge. In this context, adding $n^{2-\alpha}$ random edges means
we cut the cycle into arcs at $k = 2n^{2-\alpha}$ random points.
Asymptotically this is equivalent to splitting the unit interval by $k-1$
i.i.d.\ uniform variables (in terms of the resulting lengths). For the
length $l$ of the largest gap it is known that
$$c_1\log (k-1)/(k-1)< l < c_2\log (k-1)/(k-1)$$
a.a.s. See Slud
\cite{slud1978entropy} or Devroye \cite{devroye1981laws} for details.
Therefore the number of nodes in the
longest empty arc $A$ is a.a.s.\ at least
$$nl \ge cn\frac{\log k}{k} = cn\frac{(2-\alpha)\log n +\log 2}{2n^{2-\alpha}} =
cn^{\alpha-1}\log n + O(n^{\alpha-1}).$$
Consequently we can use a similar estimate as before:
$$\Phi
\le
\frac{Q(A,A^C)}{\pi(A)\pi(A^C)} \le
\frac{2n^{-1}}{cn^{\alpha-2}\log n\cdot 1/2}=c\frac{n^{1-\alpha}}{\log n}.$$

For the proof of the lower bound we intend to follow the same idea as
for Theorem \ref{THM:M1COND}, but a few things have to be updated. First of all, there might be nodes which have multiple long
range edges. For the graph on $\tilde{\cX}$ we want the long range
edges to form a random matching. Thus we include multiple copies of such
a node and randomly distribute the long range edges among them, see Figure \ref{FIG:GRAPHREDM2}.

\begin{figure}[ht]
\centering
\includegraphics[width=0.7\columnwidth]{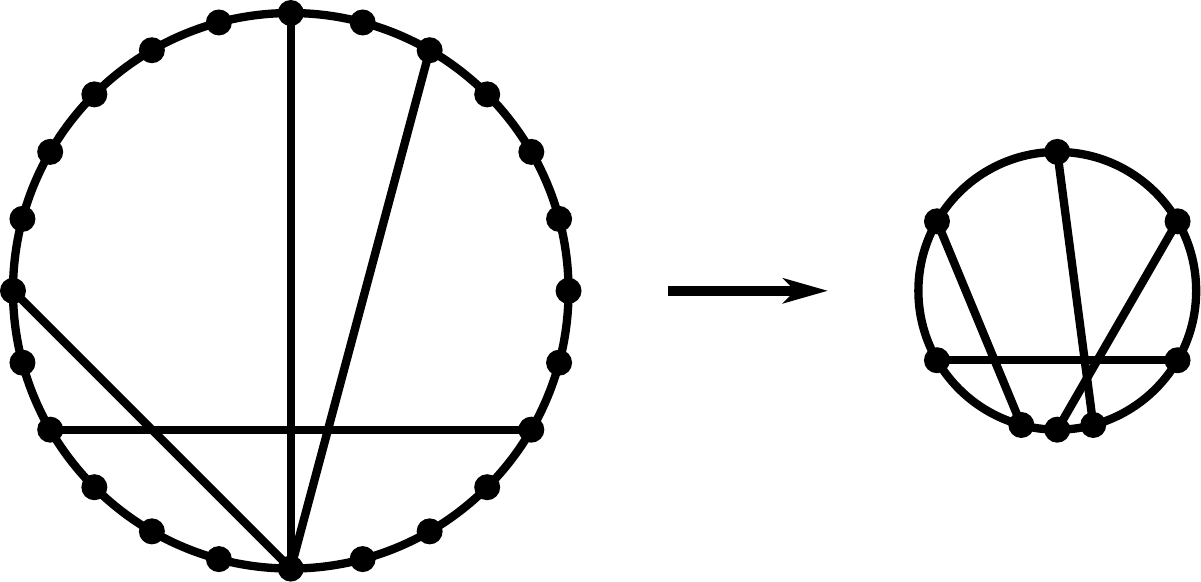}
\caption{Reducing M2 graphs}\label{FIG:GRAPHREDM2}
\end{figure}

We use similar inequalities to those in the proof of Theorem
\ref{THM:M1COND}. There are still $n^{2-\alpha}$ long range edges,
thus the reduced graph has $2n^{2-\alpha}$ nodes again. Once again, we
use that the stationary probability of individual points scale up by
$2n^{2-\alpha}/n$. Also, a boundary edge of $S$ becomes at most one
boundary edge of $\tilde{S}$ of the same type, and thus with the same
transition probability. We thus arrive at
$$
Q(S,S^C) > 2n^{1-\alpha}Q(\tilde{S},\tilde{S}^C).
$$
This time, the collapsed arcs are not necessarily of the same
length. Still, we can use $cn^{\alpha-1}\log n$ as an upper bound as
we have shown before. This results in a weakened version of the second inequality:
$$
\pi(S) < c\log n \tilde{\pi}(\tilde{S}).
$$
For the third inequality we
use the same trick as before, swapping $S$ with $S^C$ if necessary to ensure
$\tilde{S}^C$ is large. We get again
$$
\pi(S^C) < 2\tilde{\pi}(\tilde{S}^C).
$$
Joining these inequalities yields
$$
\Phi_S \ge 8c \frac{n^{1-\alpha}}{\log n} \tilde{\Phi}_{\tilde{S}}.
$$
To get the conductance we have to optimize over $S$:
$$
\Phi = \min_S \Phi_S \ge \min_S 8c \frac{n^{1-\alpha}}{\log n}
\tilde{\Phi}_{\tilde{S}} \ge \min_{\tilde{S}} 8c \frac{n^{1-\alpha}}{\log n}
\tilde{\Phi}_{\tilde{S}} = 8c \frac{n^{1-\alpha}}{\log n} \tilde{\Phi}.
$$
Let us point out the subtle detail that we might not encounter all
possible $\tilde{S}$ as a contraction of some $S$. But when we
increase the set on which we minimize by including all $\tilde{S}$,
the minimum can only decrease, this confirms the inequality.
We use Equation \ref{EQ:TILDEPHIBOUND} again for $\tilde{\Phi}$ to conclude the proof.
\end{proof}

\begin{theorem}
\label{THM:M3COND}
Consider a graph from model M3. The conductance of the homogeneous
chain on this graph satisfies the following inequality a.a.s.:
$$c_1d(\alpha)^{-1}\frac{n^{1-\alpha}}{\log
  n}<\Phi<c_2\frac{n^{1-\alpha}}{\log n}.$$
\end{theorem}
\begin{proof}
For this proof we use a more direct approach, partially based on ideas
from Durrett \cite{durrett:rgd}.
Let us start with the lower bound.
For any $S\subset \cX$, $|S|\le n/2$ we have
\begin{equation}
\label{EQ:M3INITIALSTEP}
\Phi(S) = \frac{Q(S,S^C)}{\pi(S)\pi(S^C)}\ge \frac{Q(S,S^C)}{\pi(S)}\ge\frac{cd(\alpha)^{-1}\frac{|\partial
  S|}{n}}{\frac{|S|}{n}}=cd(\alpha)^{-1}\frac{|\partial S|}{|S|},
\end{equation}
where $\partial S$ is the set of edges between $S$ and $S^C$.
We have to ensure this is large enough for all possible
subsets $S$.
Let us fix $s=|S|\le n/2$ and the number of disjoint intervals $l$ it
consists of. We focus at only these subsets at once.

We can estimate the number $k$ of possible subsets in the following way:
$$k\le
\binom{n}{l}\binom{s-1}{l-1}<
\binom{n}{l}\binom{s}{l}
.$$
The first binomial coefficient counts how we can choose the starting
points of the intervals, the second distributes the total length of
$s$ among them. To continue, we use the following inequality:
$$\binom{m}{t}\le\left(\frac{me}{t}\right)^t.$$
For $k$ this gives us
$$k \le
\left(\frac{ne}{l}\right)^{l} \left(\frac{se}{l} \right)^{l} \le \exp
\left(l\left(\log\frac n l + \log\frac s l +2\right)
\right)<\exp(4l\log n) .$$

The outgoing edges from $S$ are partially edges of the cycle at interval
boundaries and partially long range edges.
We have $2l$ edges at the interval boundaries and the
number of long range edges $L$ follows a $Binom(s(n-s),2n^{-\alpha})$
distribution. According to Equation \ref{EQ:M3INITIALSTEP} a subset
violates the conductance bound we proposed if
$$\fixcu\frac{n^{1-\alpha}}{\log n}>\frac{|\partial
  S|}{|S|}=\frac{L+2l}{s}.$$
We introduce the new notation $\fixcu$ because its value is important, as we will see.
The probability of this violation to happen for a certain set $S$ can be written in
the following way:
$$p=P\left( L
<s\fixcu\frac{n^{1-\alpha}}{\log n}-2l\right).$$
Let us introduce the temporary notation $r=s\fixcu n^{1-\alpha}/\log n-2l$. If $r\le
0$, then the above probability is 0, and we are done. If
not, then we have the implied inequality
\begin{equation}
\label{EQ:NODESINTERVALS}
s\fixcu n^{1-\alpha}>2l\log n.
\end{equation}
In this
case we have to find an upper bound on $p$.
First using $s\le n/2$ we see
$$p = P(Binom(s(n-s),2n^{-\alpha}) < r) \le P(Binom(sn/2,2n^{-\alpha})
< r).$$
We are going to use the following version of Chernoff's inequality, see e.g.,
Mitzenmacher and Upfal \cite{mitzenmacher2005probability}:
$$P(Binom(N,q) < (1-\eta)Nq) \le \exp(-Nq\eta^2/2),$$
which holds for $\eta\in (0,1)$.
In our case we have $N=sn/2,~q=2n^{-\alpha}$ and $\eta = 1 -
r/(sn^{1-\alpha})$, therefore the inequality gives
$$p \le \exp\left(-\frac{1}{2}sn^{1-\alpha}\left(1-\frac{r}{sn^{1-\alpha}}\right)^2\right).$$
We may simplify the squared term using the positivity of $\fixcu, l$,
for $n$ large enough:
$$\left(1-\frac{r}{sn^{1-\alpha}}\right)^2 =
\left(1-\frac{\fixcu}{\log n} + \frac{2r}{sn^{1-\alpha}}\right)^2 \ge
\left(1-\frac{\fixcu}{\log n}\right)^2\ge 1-2\fixcu.$$
Substituting this to the inequality above we get
$$
p \le \exp\left(\left(\fixcu - \frac{1}{2}\right)sn^{1-\alpha}\right).
$$
Now let us collect all subsets $S$ of $s$ nodes and $l$ intervals. The
probability that there is one which violates the conductance is at
most $kp$. Using Equation \ref{EQ:NODESINTERVALS} we have an upper
bound for $k$,
$$\log k < 4l\log n<(2\fixcu)sn^{1-\alpha}.$$
Let us join our previous estimates. For $n$ large enough we have
$$
\log(kp) < (2\fixcu)sn^{1-\alpha} +
\left(\fixcu - \frac{1}{2}\right)sn^{1-\alpha} =
\left(3\fixcu-\frac{1}{2}\right)sn^{1-\alpha}.
$$
For $\fixcu\le 5/18$ we get a coefficient at most $-1/6$. From Equation
\ref{EQ:NODESINTERVALS} again,
$$-\frac{1}{6}sn^{1-\alpha}<-\frac{l}{3\fixcu}\log n.$$
Here we need $\fixcu\le 1/9$ to get at most $-3\log n$. After all,
with the proper $\fixcu$ we end
up with
$$kp<\frac{1}{n^3}.$$

It is only left to sum over all possible $s$ and $l$
values. This introduces an extra $n^2$ term, but the probability remains
asymptotically 0. In the end we see the lower bound on the conductance is false only with
asymptotically vanishing probability.

Let us now turn our attention to the upper bound. If we find an arc $A$ that is at
least $c n^{\alpha-1}\log n $ long with no long range edges
going out of it then we
can use the same estimate as before:
$$\Phi\le
\frac{Q(A,A^C)}{\pi(A)\pi(A^C)} \le
\frac{2n^{-1}}{cn^{\alpha-2}\log n\cdot 1/2}=c\frac{n^{1-\alpha}}{\log n}.$$
Again, we have to be careful with the constants. We will search for an
arc at least $\fixco n^{\alpha-1}\log n$ long, and we will specify
$\fixco$ later.
To do this, let us split the cycle into arcs of length $b =
\fixco n^{\alpha-1}\log n$. We define a graph on these arcs, we connect two
of them if there is any long range edge between them. According to the
independence of the edges this is in fact
an Erd\ho s-R\'enyi random graph. Our goal translates to finding an
isolated node in it.

For a sequence of Erd\ho s-R\'enyi graphs on $m$ nodes with edge
probability $q$ it is known
\cite{erdos1959random} that they have isolated nodes a.a.s.\ if
$m\ra\infty$ but $q < (1-\varepsilon)\frac{\log m}{m}$ for some fixed
$\varepsilon > 0$.
In our
case the number of nodes is
$$m = \frac n b = \frac{n^{2-\alpha}}{\fixco\log n}.$$
We can bound the edge probability in the new graph by adding up the
appropriate edge probabilities in the original graph:
$$q \le b^22n^{-\alpha}=2\fixco^2n^{\alpha-2}\log^2 n.$$
We have to compare this quantity with the following:
$$\frac{\log m}{m} = \fixco n^{\alpha-2}\log n((2-\alpha)\log n - \log \fixco -
\log\log n).$$
The major term is the first one, which is fortunately of the same
order as $q$. In order to have an isolated node a.a.s. we simply need
\begin{align*}
2\fixco^2&<\fixco(2-\alpha),\\
\fixco&<1-\frac{\alpha}{2}.
\end{align*}
There was no other restriction on $\fixco$ apart from being positive so we
can choose it to satisfy this last inequality. This concludes the proof.
\end{proof}

\section{Mixing time bounds}
\label{SEC:RANDMIXTIME}

Let us now move on to estimate the mixing time itself. The first
result is a lower bound based on a previous result for cycles
without added edges.
\begin{proposition}
Consider the graph model M1, let us also assume the nodes with long range edges are equidistant
from each other. Then for any homogeneous chain,
$$cn^{2\alpha-2} \le t_{\rm mix}.$$
\end{proposition}
\begin{proof}
Observe that we can ``wind up'' the chain around a cycle of
$n^{\alpha-1}/2$ nodes so that long range edges become loop edges, see
Figure \ref{FIG:WINDUP}.
\begin{figure}[ht]
\centering
\includegraphics[width=0.7\columnwidth]{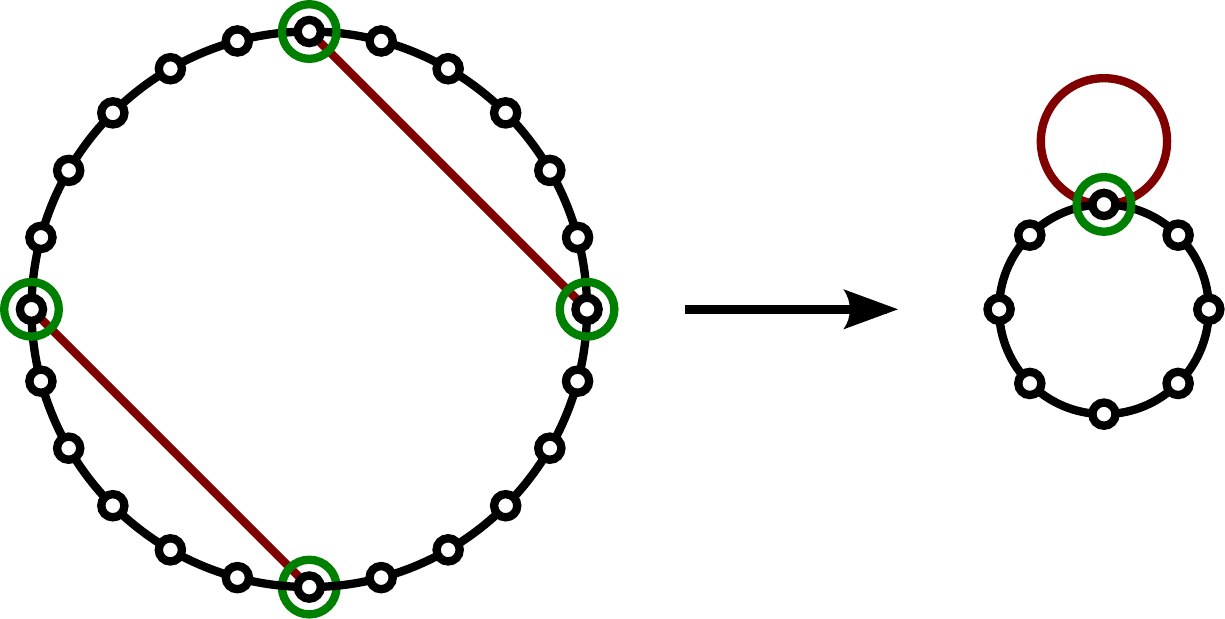}
\caption{Reducing M1 graphs}\label{FIG:WINDUP}
\end{figure}
Let us choose any starting distribution on the original chain. It is
easy to see that we get the same if we project the starting distribution on
the reduced graph and run the Markov chain there or
if we run the original Markov chain and project the resulting distribution.
Consequently the lower bound from Theorem \ref{THM:SLOWMIX} for the mixing
time of the reduced graph is also valid for the mixing time
of the original chain.
\end{proof}

The same claim is true if the nodes with long range edges are not
exactly equidistant, but the Markov chain is reversible.
\begin{proposition}
\label{PRP:M1REVLOW}
Consider the graph model M1. Then for any reversible homogeneous chain,
$$cn^{2\alpha-2} < t_{\rm mix}.$$  
\end{proposition}
\begin{proof}
By the definition of the graph model M1 there are arcs at least $cn^{\alpha-1}$ long without a long range edge.
Let us now focus only on one of these arcs.
If the Markov chain is initialized at the center of the arc, it stays
within the arc for at least
$cn^{2\alpha-2}$ steps with probability $1-\delta$. For
small $\delta>0$, this ensures mixing did not yet happen,
consequently $cn^{2\alpha-2}$ is a lower bound on the mixing
time.
\end{proof}

Using Theorem \ref{THM:LOVSIMON} together with the conductance bound
Theorem \ref{THM:M1COND} for M1 graphs we can complement this result
with an upper bound.
\begin{theorem}
\label{THM:M1HOMO}
Consider the graph model M1. The mixing time of the homogeneous lazy chain on
such graphs satisfies the following inequality a.a.s.:
$$c_1n^{\alpha-1} < t_{\rm mix} < c_2d(\alpha)^2 n^{2\alpha-2}\log
n.$$
\end{theorem}

For graphs from the model M2 we formulate bounds separately for
reversible and non-reversible Markov chains.
\begin{theorem}
\label{THM:M2HOMO}
Consider the graph model M2. The mixing time of the reversible homogeneous chain on
such graphs satisfies the following inequality a.a.s.:
$$c_1n^{2\alpha-2}\log^2 n < t_{\rm mix} < c_2d(\alpha)^2 n^{2\alpha-2}\log^3
n.$$
In the case of lazy non-reversible homogeneous chains this changes to
$$c_1n^{\alpha-1}\log n < t_{\rm mix} < c_2d(\alpha)^2 n^{2\alpha-2}\log^3
n.$$
\end{theorem}
\begin{proof}
The upper bounds and the weaker lower bounds follow by combining Theorem
\ref{THM:LOVSIMON} with the conductance bounds Theorem \ref{THM:M2COND}.
The sharper bounds for reversible chains follow the same way as for
Proposition \ref{PRP:M1REVLOW}. This time the longest arc without a
long range edge is at least $cn^{\alpha-1}\log n$ long a.a.s.\ as shown
during the proof of Theorem \ref{THM:M2COND}.
\end{proof}

In a similar way we can acquire mixing time bounds for graphs from the
model M3.
\begin{theorem}
\label{THM:M3HOMO}
Consider the graph model M3. The mixing time of the reversible homogeneous chain on
such graphs satisfies the following inequality a.a.s.:
$$c_1n^{2\alpha-2}\log^2 n < t_{\rm mix} < c_2d(\alpha)^2 n^{2\alpha-2}\log^3
n.$$
In the case of lazy non-reversible homogeneous chains this changes to
$$c_1n^{\alpha-1}\log n < t_{\rm mix} < c_2d(\alpha)^2 n^{2\alpha-2}\log^3
n.$$
\end{theorem}
\begin{proof}
Again, we use Theorem \ref{THM:LOVSIMON} and the appropriate
conductance estimate, now Theorem \ref{THM:M3COND}. We also want to
reuse the claim based on the existence of a long arc without long
range edges. However, during the proof of Theorem \ref{THM:M3COND} we
only showed that there is an arc of length $cn^{\alpha-1}\log n$ such
that there is no long range edge going \emph{out} of it. The long
range edges going within the arc are independent from the ones going
out, the probability of having none within the arc is
$$(1-2n^{-\alpha})^{\ds{(cn^{\alpha-1}\log n)^2}} = (1-2n^{-\alpha})^{\ds{n^\alpha
  c\frac{\log^2 n}{n^{2-\alpha}}}}>e^{\ds{-3c\frac{\log^2 n}{n^{2-\alpha}}}}$$
This is 1 in the limit, consequently the arc we have chosen does not have any long
range edge at all a.a.s. Therefore we can apply the same reasoning as before.
\end{proof}

These results allow us to have an insight on the order of magnitude of
the mixing time. In some cases we know the polynomial part exactly and
have a difference only in the logarithmic part. Other studies, like 
Addario-Berry and Lei \cite{addarioberry:swnmixing2012}
or Krivelevich, Reichman and Samotij \cite{krivelevich2013smoothed}
suggest that it might be possible to reduce the $\log^3n$ terms in the
upper bounds to $\log^2n$ but this is left as future work.

The bounds we got for reversible chains provide reasonably tight estimates. For
non-reversible chains it is still unclear where the mixing time really
is between these bounds.

\section{Conclusions and future work}
\label{SEC:CONC}

In the case of reversible chains we have obtained bounds on the
mixing times for the random graphs of models M1, M2, M3.
All these bounds are of the form $cn^{2\alpha-2}\log^\delta n$
with $\delta$ differing by one between the lower and upper bound
for each
specific choice of model parameters.
Closing this gap is left as a future work which could be within
reach by borrowing techniques from Addario-Berry and Lei \cite{addarioberry:swnmixing2012}
or Krivelevich, Reichman and Samotij \cite{krivelevich2013smoothed}.
In the dominant part $n^{2\alpha-2}$ the exponent may take on all possible values
between $0$ and $2$. The limiting case of $n^0$ is known to correspond
to the case of Small World Network of Newman et al.,
see \cite{durrett:rgd}, \cite{addarioberry:swnmixing2012}, the case
of $n^2$ corresponds to the case of a plain cycle, see \cite{gb:ringmixing2011}.

The situation is more diverse for non-reversible
Markov chains.
For homogeneous M1 chains with some additional
restrictions, we have shown that the mixing time does not decrease
compared to reversible Markov chains, having the same $n^{2\alpha-2}$
lower bound as before. On the other hand, in general the lower bounds
drop to $cn^{\alpha-1}\log^\delta n$. This indicates the possibility of having significantly
lower mixing times. Indeed, simulations suggest that there is a considerable gain for non-reversible chains.
In Figure \ref{FIG:MIXINGSIM} we plot the mixing times of
homogeneous reversible and non-reversible chains on several graphs
coming from model M2 with $\alpha=1.5$.

\begin{figure}[ht]
\centering
\includegraphics[width=0.75\columnwidth]{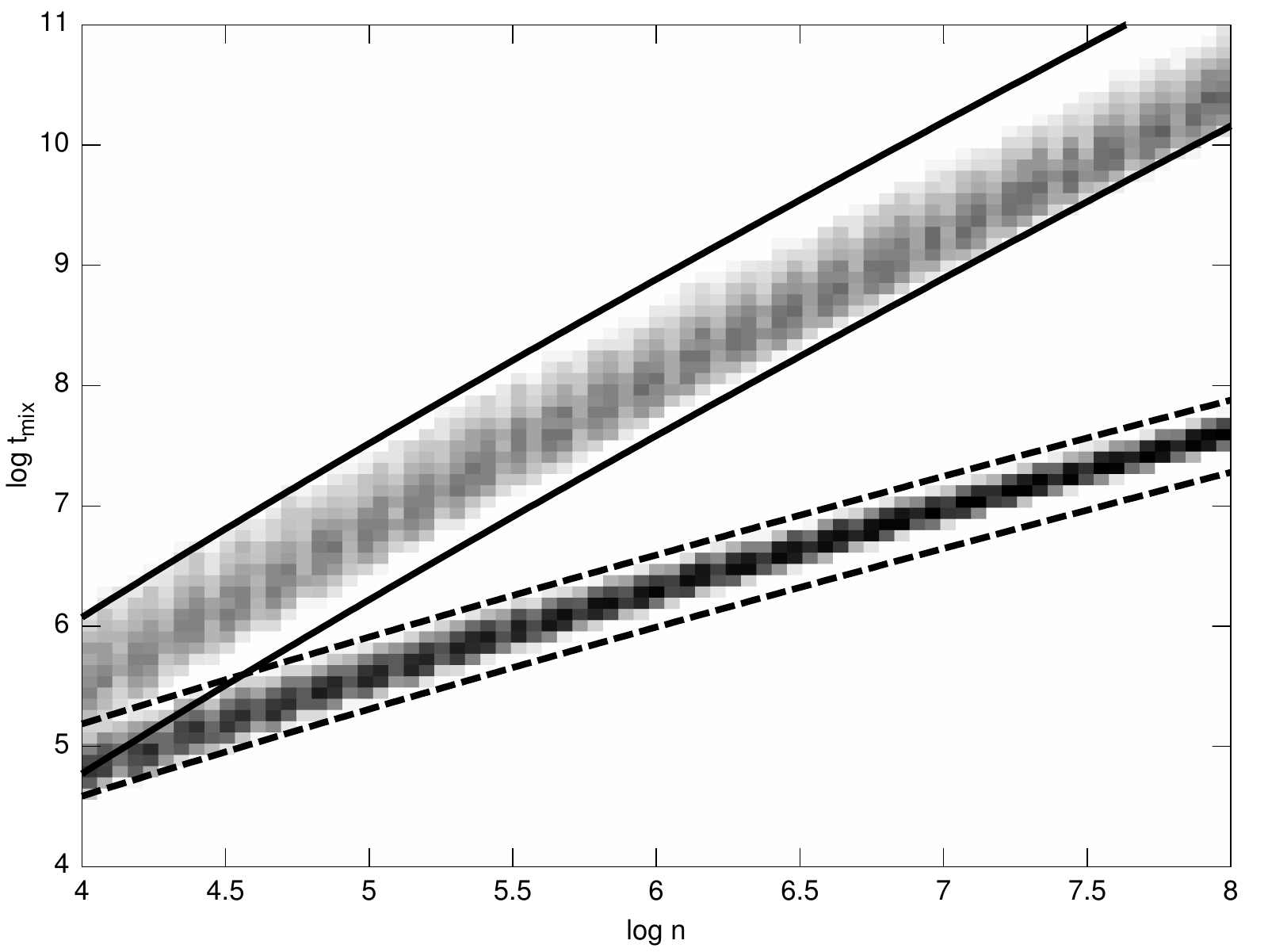}\hspace{0.06\columnwidth}
\caption{Log-log plot for mixing times of homogeneous M2 chains}\label{FIG:MIXINGSIM}
\end{figure}

This is a log-log scaled histogram using more than $70000$ random graphs in
total. Because we aim to understand the typical behavior, we discarded the
lowest and highest $5\%$ for each graph size $n$. We see two clusters,
the upper one contains the mixing times of all the reversible chains,
while the lower one contains that of all non-reversible chains.
For comparison, we include two solid lines corresponding to 
$cn\log^2n$ and two dashed lines indicating $c\sqrt{n}\log n$.

A challenging open problem is the clarification of this decrease of
mixing times of non-reversible M2 chains. In general, it would be
interesting to find other connectivity graphs,
where introducing non-reversible Markov chains offers strong speedup
compared to reversible ones, without changing the underlying graph.

A further interesting research direction might be to extend the results for time-inhomogeneous Markov
chains. In the case of the cycle, when every transition matrix is doubly stochastic and
also reversible, it is easy to show that the mixing time is at least of the order of
$n^2$. However, the case of doubly stochastic but non-reversible transition
matrices is still open, it is unclear if the result of
\cite{gb:ringmixing2011} can be extended to this case. On the other
hand, if we relax the condition on the transition matrices by not
requiring them to be doubly stochastic, we can significantly
improve the mixing time. In particular it is known that the mixing
time can be as low as $n$. 

\section*{Acknowledgments}
I would like to express my thanks to M\'arton Isp\'any for his insightful
and encouraging questions.

\bibliographystyle{siam}
\bibliography{ringmixing,swn,other}

\end{document}